\documentclass[12pt]{amsart}
\usepackage{amsmath,amsthm,amsfonts,amssymb,amscd,euscript}
\newlength{\defbaselineskip}
\theoremstyle{plain}
\newtheorem{thm}{Theorem}

\newtheorem{lem}[thm]{Lemma}
\newtheorem{cor}[thm]{Corollary}

\newtheorem{prop}[thm]{Proposition}
\newtheorem{conj}[thm]{Conjecture}

\theoremstyle{definition}
\newtheorem{defn}[thm]{Definition}
\newtheorem{defnprop}[thm]{Definition-Proposition}

\newtheorem{exmp}[thm]{Example}

\newtheorem{remark}[thm]{Remark}
\newtheorem{question}[thm]{Question}

\numberwithin{equation}{section}

\def\C{{\mathbb C}}
\def\N{{\mathbb N}}
\def\eg{\medskip\noindent \textit{Example: }}

\def\bmx{\begin{bmatrix}}
\def\emx{\end{bmatrix}}

\newcommand{\boxs}[1]
{ \multiput(#1)(10,0){2}
 {\line(0,10){10}}
\multiput(#1)(0,10){2}
 {\line(10,0){10}}
}

\address{Department of Mathematics, Purdue University, West Lafayette, IN 47907}
\email{{\tt kyungl@purdue.edu}}
\thanks{Research of the first author partially supported by NSF grant DMS
0901367}

\address{Department of Mathematics, University of Illinois at Urbana Champaign, Urbana, IL 61801}
\email{{\tt llpku@math.uiuc.edu}}

\begin{document}
\title[Diagonal ideal]{Notes on a minimal set of generators for the radical ideal defining the diagonal locus of $(\C^2)^n$}
\author{Kyungyong Lee and Li Li}
\maketitle

\begin{abstract}
We develop several techniques for the study of the radical ideal $I$ defining the diagonal locus of $(\C^2)^n$. Using these techniques, we give combinatorial construction of generators for $I$ of certain bi-degrees.
\end{abstract}

\section{Introduction}

\subsection{Overview}
Fix a positive integer $n$. Consider $n$-tuples of ordered points $\{(x_i,y_i)\}_{1\le i\le n}$ in the plane $\C^2$. The set of all $n$-tuples forms
an affine space $(\C^2)^n$ with coordinate ring $\C[\textbf{x},\textbf{y}]=\C[x_1,y_1,...,x_n,y_n]$. The symmetric group $S_n$ acts on
$\mathbb{C}[\textbf{x},\textbf{y}]$ by permuting the coordinates in $\textbf{x},\textbf{y}$ simultaneously, that is,
$$
\sigma(x_{j}):=x_{\sigma(j)},\quad \sigma(y_{j}):=y_{\sigma(j)}\quad\text{ for }\sigma \in S_n.
$$
\begin{defn}
A polynomial $f \in \mathbb{C}[\textbf{x},\textbf{y}]$ is called \emph{alternating} if
$$
\sigma(f)=\text{sgn}(\sigma)f \,\,\,\,\,\text{\label{definitionofstaircaseform}for all }\sigma \in S_n.
$$
Define $\mathbb{C}[\textbf{x},\textbf{y}]^{\epsilon}$ to be the vector space of alternating polynomials in  $\mathbb{C}[\textbf{x},\textbf{y}]$.
\end{defn}

The vector space $\mathbb{C}[\textbf{x},\textbf{y}]^{\epsilon}$ has a well-known basis which we describe below.
Denote by $\N$ the set of nonnegative integers. Let $\mathfrak{D}$ be the set of subsets $D=\{(\alpha_1,\beta_1),...,(\alpha_n,\beta_n)\}$ of $\N\times\N$. For $D\in\mathfrak{D}$, define
$$\Delta(D):=
\det\begin{bmatrix}
     x_1^{\alpha_1}y_1^{\beta_1}&x_1^{\alpha_2}y_1^{\beta_2}&...&x_1^{\alpha_n}y_1^{\beta_n}\\
      \vdots&\vdots &\ddots &\vdots\\
    x_n^{\alpha_1}y_n^{\beta_1}&x_n^{\alpha_2}y_n^{\beta_2}&...&x_n^{\alpha_n}y_n^{\beta_n}\\
   \end{bmatrix}
$$ (by abuse of notation we also use $\Delta(D)$ to denote the above square matrix).
Then $\{\Delta(D)\}_{D\in\mathfrak{D}}$ forms a basis
for the $\C$-vector space $\mathbb{C}[\textbf{x},\textbf{y}]^{\epsilon}$.

The radical ideal  $I$ that defines the diagonal locus of $(\C^2)^n$ is
$$I=\bigcap_{1\leq i<j\leq n}(x_i-x_j, y_i-y_j).$$
A famous theorem of Haiman asserts the following:
\begin{thm}\emph{\cite[Corollary 3.8.3]{H:hil}}
The ideal generated by the alternating polynomials in $\mathbb{C}[\textbf{x},\textbf{y}]$ agrees with $I$.
\end{thm}
Haiman's theorem immediately implies that the ideal $I$ is generated by $\{\Delta(D)\}_{D\in \mathfrak{D}}$. He has also proved the following theorem, which asserts that the number of minimal generators of $I$ is equal to the $n$-th Catalan number.
\begin{thm}\emph{(\cite[p393]{van})}
$\dim_{\C}\,I/(\mathbf{x},\mathbf{y})I= \displaystyle\frac{1}{n+1}{2n\choose n}$.
\end{thm}
Let $M=I/(\mathbf{x},\mathbf{y})I$. The space $M$ is doubly graded as $M=\oplus_{d_1,d_2}M_{d_1,d_2}$. The $t,q$-analog of the Catalan number is defined as
$$C_n(q,t)=\sum_{d_1,d_2}t^{d_1}q^{d_2}\dim M_{d_1,d_2}.$$
By a simple algebraic argument, giving a minimal set of generators of $I$ is equivalent to giving a basis of $M$, and it suffices to find bases for all graded pieces $M_{d_1,d_2}$.  It is then natural to ask the following question:

\begin{question}\label{question}
Given a bi-degree $(d_1,d_2)$, is there a combinatorially significant construction of the basis for each $M_{d_1,d_2}$?
\end{question}

On one hand, a combinatorial study of $C_n(q,t)$ gives us hints to construct an explicit basis for $M_{d_1,d_2}$; on the other hand, a good understanding of $M_{d_1,d_2}$ helps to study $C_n(q,t)$. For example, if we can show that some of the combinatorially significant elements in $M_{d_1,d_2}$ are linearly independent, then we can give lower or upper bounds for the coefficients appeared in $C_n(q,t)$. This idea is developed further in our subsequent paper \cite{LL2}.

Of course Question \ref{question} is vague and only gives the guideline of study. Here is one way to make it precise, which is one motivation for us to study Question \ref{question}.
Define
$$ \Lambda:=\{\lambda=(\lambda_1,\dots,\lambda_n)\,|\, \lambda_1\ge\lambda_2\cdots\ge\lambda_{n-1}\ge\lambda_n=0,\quad \lambda_i\le n-i, \forall 1\le i\le n\},
$$
$$\aligned
 &a(\lambda):=\sum_{i=1}^n(n-i-\lambda_i),\quad\forall \lambda\in\Lambda,\\
&b(\lambda):=\#\{(i,j)\,|\, i<j, \, \lambda_i-\lambda_j+i-j\in\{0,1\}\}, \quad\forall \lambda\in\Lambda,\endaligned$$
A surprising combinatorial interpretation for $C_n(q,t)$ found by Garsia and Haglund (\cite{GH:pos}, \cite{GH:proof}) asserts that $$C_n(q,t)=\sum_{\lambda\in\Lambda}q^{a(\lambda)}t^{b(\lambda)}.$$
It is then natual to ask if there is actually an explicit construction of the basis of $M_{d_1,d_2}$ hidden behind the above combinatorial interpretation. Indeed, a more specified question is posed by Haiman (\cite{H04}):

Is there a rule to associate to each $\lambda\in\Lambda$ an element $D(\lambda)\in\mathfrak{D}$ such that $\deg_\textbf{x}\Delta({D(\lambda)})=a(\lambda)$, $\deg_\textbf{y}\Delta({D(\lambda)})=b(\lambda)$, and the set $\{\Delta({D(\lambda)})\}_{\lambda\in\Lambda}$ generates $I$?

\subsection{Techniques and main result}
In the study of above questions, we found the following three linear relations that turn the questions into combinatorial games, and lead to a combinatorial construction of bases of $M_{d_1,d_2}$ for certain bi-degrees $(d_1,d_2)$. First we introduce some notations.

$\bullet$ For $D=\{P_1,\dots,P_n\}\in\mathfrak{D}$ where $P_i=(\alpha_i,\beta_i)$, define $|P_i|=\alpha_i+\beta_i$.

$\bullet$ For two homogeneous polynomials $f,g\in I$ of the same degree $d$, we denote $$f\sim g$$ if $f$ and $g$ are equivalent modulo the ideal $I_{<d}$. In other words, under the quotient map $I\to M=I/(\textbf{x},\textbf{y})I$, the image of $f$ equals the image of $g$.

\noindent\textit{Relation 1 (Transfactor Lemma \ref{lem:transfactor})}.
Given positive integers $1\le i\neq j\le n$ such that  $|P_i|=i-1$, $|P_{i+1}|=i$, $|P_j|=j-1$, $|P_{j+1}|=j$, $\beta_i>0$, $\alpha_j>0$   (we assume $|P_{n+1}|=n$).
Let $D'$ be obtained from $D$ by moving $P_i$ to southeast and $P_j$ to northwest, i.e.
$$D'=\{P_1, \dots, P_{i-1},
P_i+(1,-1),P_{i+1},\dots, P_{j-1},
P_j+(-1,1),P_{j+1},\dots,P_n\}.$$
Then $\Delta(D)\sim \Delta(D')$.

\eg $n=9, i=2, j=6$.  $$\setlength{\unitlength}{1.5pt}
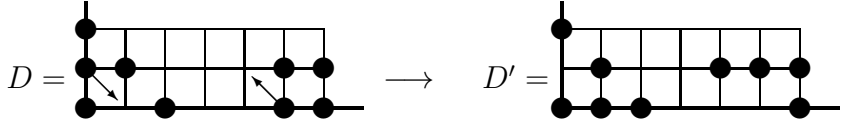
    \begin{picture}(120,30)(0,-5)
    \put(-20,5){$D=$}
    \put(0,0){\circle*{5}}
    \put(0,10){\circle*{5}}\put(0,20){\circle*{5}}\put(10,10){\circle*{5}}
    \put(20,0){\circle*{5}}
    \put(50,0){\circle*{5}}\put(50,10){\circle*{5}}\put(60,0){\circle*{5}}\put(60,10){\circle*{5}}
    \boxs{0,0}\boxs{10,0}\boxs{20,0}\boxs{30,0}\boxs{40,0}
    \boxs{50,0}
    \boxs{0,10}\boxs{10,10}\boxs{20,10}\boxs{30,10}\boxs{40,10}
    \boxs{50,10}
    \linethickness{1pt}\put(0,0){\line(0,1){27}}
    \linethickness{1pt}\put(0,0){\line(1,0){70}}
    \put(0,10){\vector(1,-1){8}}
    \put(50,0){\vector(-1,1){8}}
    \put(75,5){$\longrightarrow$}
     \end{picture}
    \begin{picture}(100,30)(0,-5)
    \put(-20,5){$D'=$}
    \put(0,0){\circle*{5}}
    \put(10,0){\circle*{5}}\put(0,20){\circle*{5}}\put(10,10){\circle*{5}}
    \put(20,0){\circle*{5}}
    \put(40,10){\circle*{5}}\put(50,10){\circle*{5}}\put(60,0){\circle*{5}}\put(60,10){\circle*{5}}
    \boxs{0,0}\boxs{10,0}\boxs{20,0}\boxs{30,0}\boxs{40,0}
    \boxs{50,0}
    \boxs{0,10}\boxs{10,10}\boxs{20,10}\boxs{30,10}\boxs{40,10}
    \boxs{50,10}
    \linethickness{1pt}\put(0,0){\line(0,1){27}}
    \linethickness{1pt}\put(0,0){\line(1,0){70}}
    \end{picture}
    $$

\noindent\textit{Relation 2 (Permuting Lemma \ref{lem:MinorsPermutingLemma})}.
Given positive integers $h, \ell$ and $m$ such that
$2\le h<h+\ell+m\le n+1$,  $|P_h|=h-1,|P_{h+\ell}|=h+\ell-1$, $|P_{h+\ell+m}|=h+\ell+m-1$
(by convention, the last equality holds if
$h+\ell+m=n+1$) and
$\alpha_{h+\ell},...,\alpha_{h+\ell+m-1}\geq \ell$. Let $D'$ be obtained from $D$ by moving the $m$ points $P_{h+\ell},\dots,P_{h+\ell+m-1}$ to the left by $\ell$ units and moving the $\ell$ points $P_h,\dots,P_{h+\ell-1}$ to the right by $m$ units, i.e.
$$\aligned D'=\{&P_1,P_2, \dots, P_{h-1},
P_{h+\ell}-(\ell,0),P_{h+\ell+1}-(\ell,0),\dots,P_{h+\ell+m-1}-(\ell,0),\\
&P_{h}+(m,0),P_{h+1}+(m,0),\dots,P_{h+\ell-1}+(m,0),
P_{h+\ell+m},\dots,P_{n}\}.\endaligned$$
Then $\Delta(D)\sim\Delta(D')$.

\eg $n=10, h=3, \ell=4, m=3$.  $$\setlength{\unitlength}{1.5pt}
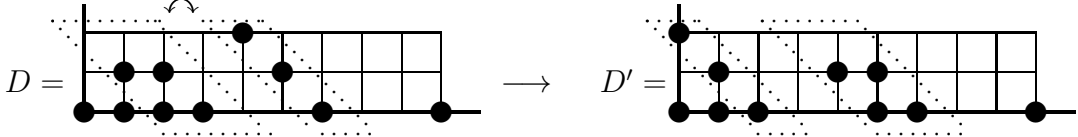
    \begin{picture}(150,30)(0,-5)
    \put(-20,5){$D=$}
    \put(0,0){\circle*{5}}\put(10,0){\circle*{5}}
    \put(10,10){\circle*{5}}\put(20,0){\circle*{5}}\put(20,10){\circle*{5}}
    \put(30,0){\circle*{5}}
    \put(40,20){\circle*{5}}\put(50,10){\circle*{5}}\put(60,0){\circle*{5}}
    \put(90,0){\circle*{5}}
    \boxs{0,0}\boxs{10,0}\boxs{20,0}\boxs{30,0}\boxs{40,0}
    \boxs{50,0}\boxs{60,0}\boxs{70,0}\boxs{80,0}
    \boxs{0,10}\boxs{10,10}\boxs{20,10}\boxs{30,10}\boxs{40,10}
    \boxs{50,10}\boxs{60,10}\boxs{70,10}\boxs{80,10}
    \linethickness{1pt}\put(0,0){\line(0,1){27}}
    \linethickness{1pt}\put(0,0){\line(1,0){100}}
    \multiput(-8,23)(2,-2){14}{\circle*{.9}}
    \multiput(18,23)(2,-2){14}{\circle*{.9}}
    \multiput(-8,23)(3,0){10}{\circle*{.9}}
    \multiput(20,-5.5)(3,0){10}{\circle*{.9}}
    \multiput(30,23)(2,-2){14}{\circle*{.9}}
    \multiput(45,23)(2,-2){14}{\circle*{.9}}
    \multiput(32,23)(3,0){5}{\circle*{.9}}
    \multiput(60,-5.5)(3,0){5}{\circle*{.9}}
   \put(21,25){$\curvearrowright$}
   \put(20,25){$\curvearrowleft$}
    \put(105,5){$\longrightarrow$}
     \end{picture}
    \begin{picture}(100,30)(0,-5)
    \put(-20,5){$D'=$}
    \put(0,0){\circle*{5}}
    \put(10,0){\circle*{5}}\put(0,20){\circle*{5}}\put(10,10){\circle*{5}}
    \put(20,0){\circle*{5}}
    \put(40,10){\circle*{5}}\put(50,0){\circle*{5}}\put(50,10){\circle*{5}}\put(60,0){\circle*{5}}
    \put(90,0){\circle*{5}}
    \boxs{0,0}\boxs{10,0}\boxs{20,0}\boxs{30,0}\boxs{40,0}
    \boxs{50,0}\boxs{60,0}\boxs{70,0}\boxs{80,0}
    \boxs{0,10}\boxs{10,10}\boxs{20,10}\boxs{30,10}\boxs{40,10}
    \boxs{50,10}\boxs{60,10}\boxs{70,10}\boxs{80,10}
    \linethickness{1pt}\put(0,0){\line(0,1){27}}
    \linethickness{1pt}\put(0,0){\line(1,0){100}}
    \multiput(-8,23)(2,-2){14}{\circle*{.9}}
    \multiput(5,23)(2,-2){14}{\circle*{.9}}
    \multiput(-8,23)(3,0){5}{\circle*{.9}}
    \multiput(20,-5.5)(3,0){5}{\circle*{.9}}
    \multiput(21,23)(2,-2){14}{\circle*{.9}}
    \multiput(45,23)(2,-2){14}{\circle*{.9}}
    \multiput(21,23)(3,0){8}{\circle*{.9}}
    \multiput(50,-5.5)(3,0){8}{\circle*{.9}}
    \end{picture}
    $$

\noindent\textit{Relation 3 (Lemma \ref{lem:powerful})}. Given positive integers $j$ and $s$. Suppose $P_{s_0}$ is the last point in $D$ satisfying $|P_i|=i-1$. Define $j=(s_0-1-|P_{s_0}|)+(s_0-|P_{s_0+1}|)+\cdots+(n-1-|P_n|)$. Suppose $|P_i|=i-1$ for $1\le i\le j+2$, $P_2=(1,0)$, $s_0\le s\le n$, and  $\alpha_s,\beta_s\ge 1$. Let
$$\aligned &D^\nwarrow=\{P_1,\dots,P_{j+1},P_{j+2}+(1,-1),P_{j+3},\dots,P_{s-1},P_s+(-1,1),P_{s+1},\dots,P_n\},\\
&D^\searrow=\{P_1,(0,1),P_3,\dots,P_{s-1},P_s+(1,-1),P_{s+2},\dots,P_n\}.\endaligned$$
Then $2\Delta(D)\sim \Delta(D^\nwarrow)+\Delta(D^\searrow)$.

\eg $n=9, i=2, j=6$.  $$\setlength{\unitlength}{1.5pt}
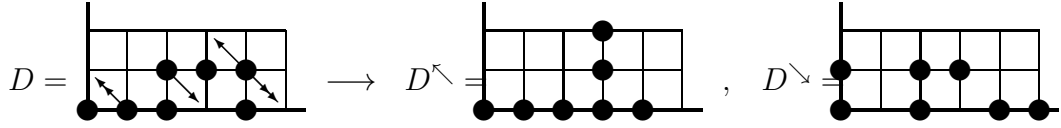
    \begin{picture}(100,30)(0,-5)
    \put(-20,5){$D=$}
    \put(0,0){\circle*{5}}\put(10,0){\circle*{5}}\put(20,0){\circle*{5}}\put(20,10){\circle*{5}}
    \put(30,10){\circle*{5}}\put(40,0){\circle*{5}}\put(40,10){\circle*{5}}
    \boxs{0,0}\boxs{10,0}\boxs{20,0}\boxs{30,0}\boxs{40,0}
    \boxs{0,10}\boxs{10,10}\boxs{20,10}\boxs{30,10}\boxs{40,10}
    \linethickness{1pt}\put(0,0){\line(0,1){27}}
    \linethickness{1pt}\put(0,0){\line(1,0){55}}
    \put(10,0){\vector(-1,1){8}}    \put(10,0){\vector(-1,1){6}}
    \put(20,10){\vector(1,-1){8}}
    \put(40,10){\vector(-1,1){8}}
    \put(40,10){\vector(1,-1){6}}    \put(40,10){\vector(1,-1){8}}
    \put(60,5){$\longrightarrow$}
     \end{picture}
     \begin{picture}(90,30)(0,-5)
    \put(-20,5){$D^\nwarrow=$}
    \put(0,0){\circle*{5}}\put(10,0){\circle*{5}}\put(20,0){\circle*{5}}\put(30,0){\circle*{5}}
    \put(30,10){\circle*{5}}\put(40,0){\circle*{5}}\put(30,20){\circle*{5}}
    \boxs{0,0}\boxs{10,0}\boxs{20,0}\boxs{30,0}\boxs{40,0}
    \boxs{0,10}\boxs{10,10}\boxs{20,10}\boxs{30,10}\boxs{40,10}
    \linethickness{1pt}\put(0,0){\line(0,1){27}}
    \linethickness{1pt}\put(0,0){\line(1,0){55}}
    \put(60,5){,}
     \end{picture}
     \begin{picture}(70,30)(0,-5)
    \put(-20,5){$D^\searrow=$}
    \put(0,0){\circle*{5}}\put(0,10){\circle*{5}}\put(20,0){\circle*{5}}\put(20,10){\circle*{5}}
    \put(30,10){\circle*{5}}\put(40,0){\circle*{5}}\put(50,0){\circle*{5}}
    \boxs{0,0}\boxs{10,0}\boxs{20,0}\boxs{30,0}\boxs{40,0}
    \boxs{0,10}\boxs{10,10}\boxs{20,10}\boxs{30,10}\boxs{40,10}
    \linethickness{1pt}\put(0,0){\line(0,1){27}}
    \linethickness{1pt}\put(0,0){\line(1,0){55}}
     \end{picture}
$$

By playing with the above three relations, we can easily obtain our main theorem.
Let us first define minimal staircase forms.
\begin{defn}\label{df2:minmal staircase form} We call $D=\{P_1,\dots,P_n\}$ a \emph{minimal staircase form} if
$|P_i|=i-1$ or $i-2$ for every $1\le i\le n$. For a minimal
staircase form $D$, let $\{i_1<i_2<\dots<i_\ell\}$ be the set of
$i$'s such that $|P_i|=i-1$, we define the \emph{partition type} of
$D$ to be the partition of (${n\choose 2}-\sum|P_i|$) consisting of
all the positive integers in the sequence
$$(i_1-1,i_2-i_1-1,i_3-i_2-1,\dots,i_\ell-i_{\ell-1}-1, n-i_\ell).$$
\end{defn}
\eg
  Let $n=8$ and $D=\{P_1,\dots,P_8\}$ satisfying
  $(|P_1|,\dots,|P_8|)=(0,1,1,2,4,4,5,6)$. Then $D$ is a minimal staircase form. The set
  $\{i\,\big{|}\, |P_i|=i-1\}$ equals $\{1,2,5\}$. The positive integers in the sequence $(1-1,2-1-1,5-2-1,8-5)$ are $(2,3)$, so the partition type of $D$ is
  $(2,3)$.
$$\setlength{\unitlength}{1.5pt}
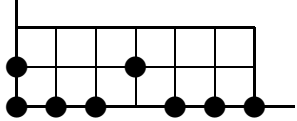
   \begin{picture}(90,30)(0,-5)
    \put(0,0){\circle*{5}}\put(0,10){\circle*{5}}\put(10,0){\circle*{5}}\put(20,0){\circle*{5}}
    \put(30,10){\circle*{5}}\put(40,0){\circle*{5}}\put(50,0){\circle*{5}}\put(60,0){\circle*{5}}
    \boxs{0,0}\boxs{10,0}\boxs{20,0}\boxs{30,0}\boxs{40,0}\boxs{50,0}
    \boxs{0,10}\boxs{10,10}\boxs{20,10}\boxs{30,10}\boxs{40,10}\boxs{50,10}
    \linethickness{1pt}\put(0,0){\line(0,1){27}}
    \linethickness{1pt}\put(0,0){\line(1,0){70}}
     \end{picture}
$$

Let $p(k)$ denote
the number of partitions of an integer $k$ and $\Pi_k$ denote the set of partitions of $k$.

\begin{thm}[Main Theorem]\label{mainthm}
Let $k$ be any positive integer. There are positive constants
$c_1=8k+5, c_2=2k+1$ such that the following holds:

For integers $n,d_1,d_2$ satisfying $n\ge c_1$, $d_1\ge c_2n$,
$d_2\ge c_2n$ and $d_1+d_2={n\choose 2}-k$, the vector space
$M_{d_1,d_2}$ has dimension $p(k)$, and the $p(k)$ elements
$$\big{\{}\textrm{a minimal staircase form of bi-degree $(d_1,d_2)$
and of partition type $\mu$} \big{\}}_{\mu\in\Pi_k}$$ form a basis
of $M_{d_1,d_2}$.
\end{thm}

\begin{remark}
Bergeron and Chen have found explicit bases for $M_{d_1,d_2}$ for
certain bi-degrees using a different method  \cite{BC}.
\end{remark}

Let $p(\delta, k)$ be the number of
partitions of $k$ into at most $\delta$ parts. It is elementary that
$p(\delta, k)$ is the same as the number of partitions of $k$ into
parts no larger than $\delta$ (for example, \cite[p.83]{hardy}). By
convention $p(\delta,0)=1$ for $\delta\ge 0$; $p(0,k)=0$ for $k>0$. We pose the following conjecture generalizing the main theorem.
\begin{conj}\label{supermainthm}
Let $k\leq n-3$ be a non-negative integer. Let $d_1, d_2$ be two non-negative integers such that $d_1+d_2=n(n-1)/2-k$. Let $\delta=\emph{min}\{d_1, d_2\}$. Then
$$\emph{dim } M_{d_1, d_2}=p(\delta, k),$$
and there is a basis that can be constructed combinatorially.
\end{conj}

\begin{remark}
The conjecture is proved in our subsequent paper \cite{LL2}.
\end{remark}

The structure of the paper is as follows. In \S2 we give the definition of staircase forms and discuss their properties. In \S3 we prove the three relations given at the beginning of \S1.2, and  at the end of this section we give the proof of the main theorem. \S4 gives a conjectural minimal set of generators as an answer to Haiman's question, which is equivalent to a conjecture of Mahir Can and Nick Loehr.

\noindent\emph{Acknowledgements}. The authors thank Jim Haglund, Alexander Woo and Alexander Yong for many suggestions and shared insights. We thank the anonymous referees for helpful and constructive suggestions and comments on the paper.

\section{Asymptotic behavior of $t,q$-Catalan numbers}
In this section,  we first introduce the notion of staircase form and block diagonal form, which are matrices whose determinants are equivalent to $\Delta(D)$ modulo $(\mathbf{x},\mathbf{y})I$. Then we define the partition type of a staircase form. Finally we give Corollary \ref{lessthanpartitionnumber}, which is half of the main theorem.

\begin{defn}
Let $D=\{P_1,\dots,P_n\}$ where $P_i=(\alpha_i,\beta_i)$ for $1\le i\le n$. Define $s_i:=\alpha_i+\beta_i$. We say $D$ is \emph{in standard order}, if (i) $s_1\leq s_2\leq \cdots\leq s_n$, and (ii) if $s_i=s_{i+1}$ then $\alpha_i<\alpha_{i+1}$.
\end{defn}

\begin{defnprop}[Staircase form]\label{definitionofstaircaseform}
Let $D=\{(\alpha_{1},\beta_{1}), (\alpha_{2},\beta_{2}),\dots, (\alpha_{n},\beta_{n})\}\in\mathfrak{D}$. Define $s_j:=\alpha_j + \beta_j$. Define $d:=\sum_js_j=\sum_j(\alpha_j+\beta_j)$ the degree of $D$, and $(d_1,d_2):=(\sum\alpha_j,\sum\beta_j)$ the bi-degree of $D$. Define $k=n(n-1)/2-d$ the \emph{deficit} of $D$.  Denote by $I_{<d}$ the ideal of $\C[\textbf{x},\textbf{y}]$ generated by homogeneous elements of degree less than $d$ in $I$. Then there is a matrix $S$ whose (i,j)-th entry is
$$\left\{
           \begin{array}{ll}
0, &\hbox{if $i\le s_j$};\\
a_{i1}a_{i2}\cdots a_{i,s_j}  \hbox{ where $a_{i\ell}$ is either $x_i-x_\ell$ or $y_i-y_\ell$},&\hbox{otherwise,}
           \end{array}
         \right.$$
for all $1\le i,j\le n$, such that $\det S\sim\Delta(D)$  modulo $I_{<d}$.  Rearranging the columns of $S$ if necessary, we assume that  $s_1\leq s_2\leq \cdots\leq s_n$. We call the matrix $S$, or its determinant $\det S$, a \emph{staircase form} of $D$.

\end{defnprop}
\begin{proof}
The idea is to construct a matrix that is as close as possible to an Echelon form modulo $I_{<d}$.

For simplicity of notation, denote $x_{ij}=x_i-x_j$, $y_{ij}=y_i-y_j$. If $\alpha_1>0$, then the first column of the matrix $\Delta(D)$
$$[x_1^{\alpha_1}y_1^{\beta_1},\dots  , x_n^{\alpha_1}y_n^{\beta_1}]^T$$
equals to
$$x_1[x_1^{\alpha_1-1}y_1^{\beta_1},\dots , x_n^{\alpha_1-1}y_n^{\beta_1}]^T+[0,x_2^{\alpha_1-1}x^{}_{21}y_2^{b_1},\dots,x_n^{\alpha_1-1}x^{}_{n1}y_n^{\beta_1}]^T.$$
Therefore
$$\Delta(D)= x_1\det\bmx x_1^{\alpha_1-1}y_1^{\beta_1}&\cdots  &  x_1^{\alpha_n}y_1^{\beta_n}\\
\vdots&\ddots &\vdots\\
x_n^{       \alpha_1-1}y_n^{\beta_1} &\cdots  & x_n^{\alpha_n}y_n^{\beta_n} \emx+
\det \bmx 0& x_1^{\alpha_2}y_1^{\beta_2}&\cdots  &  x_1^{\alpha_n}y_1^{\beta_n}\\
x_2^{\alpha_1-1}x^{}_{21}y_2^{\beta_1}& x_2^{\alpha_2}y_2^{\beta_2}&\cdots & x_2^{\alpha_n}y_2^{\beta_n}\\
\vdots&\vdots &\ddots &\vdots\\
x_n^{\alpha_1-1}x^{}_{n1}y_n^{\beta_1}& x_n^{\alpha_2}y_n^{\beta_2} &\cdots & x_n^{\alpha_n}y_n^{\beta_n} \emx$$ The first summand is a polynomial in $I_{<d}$, so $\Delta(D)$
is equivalent to the second summand modulo $I_{<d}$. If $\alpha_1-1>0$, we write the first column of the second matrix
$$[0,x_2^{\alpha_1-1}x^{}_{21}y_2^{\beta_1},\dots,x_n^{\alpha_1-1}x_{n1}^{}y_n^{\beta_1}]^T$$
as a
sum of two vectors $$x_2[0,x_2^{\alpha_1-2}x^{}_{21}y_2^{\beta_1},\dots,x_n^{\alpha_1-2}x^{}_{n1}y_n^{\beta_1}]^T
+ [0,0, x_3^{\alpha_1-2}x^{}_{32}x^{}_{31}y_3^{\beta_1},\dots,x_n^{\alpha_1-2}x^{}_{n2}x^{}_{n1}y_n^{\beta_1}]^T.$$ Then by a similar argument as above, $\Delta(D)$ is equivalent to
$$\det \bmx 0& x_1^{\alpha_2}y_1^{\beta_2}&\cdots  &  x_1^{\alpha_n}y_1^{\beta_n}\\
0& x_2^{\alpha_2}y_2^{\beta_2}&\cdots  &  x_2^{\alpha_n}y_2^{\beta_n}\\
x_3^{\alpha_1-2}x^{}_{32}x^{}_{31}y_3^{\beta_1}& x_3^{\alpha_2}y_3^{\beta_2}&\cdots & x_3^{\alpha_n}y_3^{\beta_n}\\
\vdots&\vdots &\ddots &\vdots\\
x_n^{\alpha_1-2}x^{}_{n2}x^{}_{n1}y_n^{\beta_1}&x_n^{\alpha_2}y_n^{\beta_2} &\cdots & x_n^{\alpha_n}y_n^{\beta_n} \emx$$ modulo $I_{<d}$. If $\beta_1>0$, we can apply the similar operation. Repeating this operation, we will eventually replace the first
column by the following column vector
$$\begin{bmatrix}
 0\\
 \vdots\\
 0\\
 x_{s_1+1,1}x_{s_1+1,2}\cdots x_{s_1+1,\alpha_1}y_{s_1+1,\alpha_1+1}y_{s_1+1,\alpha_1+2}\cdots y_{s_1+1,s_1}\\
x_{s_1+2,1}x_{s_1+2,2}\cdots x_{s_1+2,\alpha_1}y_{s_1+2,\alpha_1+1}y_{s_1+2,\alpha_1+2}\cdots y_{s_1+2,s_1}\\
\vdots\\
x_{n1}x_{n2}\cdots x_{n,\alpha_1}y_{n,\alpha_1+1}y_{n,\alpha_1+2}\cdots y_{n,s_1}
\end{bmatrix}$$
where the first $\min\{s_1,n\}$ entries are 0. Note that we may use a different order of operations with respect to $x_i$ or $y_i$, and the nonzero entries in the first column result might be different.

Applying this procedure for every column, we get a matrix with $\min\{s_j,n\}$ zeros at the $j$-th column for $1\le j\le n$. Rearrange the columns such that the numbers of zeros in the columns are weakly increasing from left to right. The resulting matrix is a staircase form $S$ of $D$.
\end{proof}

\begin{cor}\label{trivialdet}
Let $D$ and $S$ be defined as in Definition-Proposition
\ref{definitionofstaircaseform}. If $s_j>j-1$ for some $1\leq j\leq
n$, then $\Delta(D)\in I_{<d}$.
\end{cor}
\begin{proof}
It is easy to see that $\text{det}S=0$.
\end{proof}

\begin{defn}\label{df:block diagonal} Let $D$ and $S$ be defined as
in Definition-Proposition \ref{definitionofstaircaseform}. Consider the set $\{j: \, s_j=j-1\}=\{r_1< r_2<\dots< r_\ell\}$ and define $r_{\ell+1}=n+1$. For  $1\le t\le \ell$, define the $t$-th block $B_t$ of $S$ to be the square submatrix of $S$ of size $(r_{t+1}-r_t)$ whose upper-left corner is the ($r_t,r_t$)-entry.  Define the \emph{block diagonal form} $B(S)$ of $S$ to be the block diagonal matrix $\hbox{diag}(B_1,\dots,B_\ell)$.
\end{defn}

It is easy to see that $\det B(S)=\det S$.

\begin{exmp}\label{examplesta}
Let $D=\{(0,0),(1,0),(0,2),(1,1),(3,1)\}$. It is in standard order.
Then
$$\Delta(D)=\left[
                                   \begin{array}{ccccc}
                                      1 & \,x_{1} &\,y_{1}^2 &x_{1} y_{1}  &\, x_{1}^{3} y_{1} \\
                                      1 & \,x_{2} &\,y_{2}^2 &x_{2} y_{2}  &\, x_{2}^{3} y_{2} \\
                                      1 & \,x_{3} &\,y_{3}^2 &x_{3} y_{3}  &\, x_{3}^{3} y_{3} \\
                                      1 & \,x_{4} &\,y_{4}^2 &x_{4} y_{4}  &\, x_{4}^{3} y_{4} \\                                                                           1 & \,x_{5} &\,y_{5}^2 &x_{5} y_{5}  &\, x_{5}^{3} y_{5} \\
                                                                            \end{array}
                                  \right],
$$
and one possible staircase form of $D$ is
$$\left[
                                   \begin{array}{ccccc}
                                      1 & \,0      &\,0            &\,0  &\, 0 \\
                                      1 & \,x_{21} &\,0            &\,0  &\, 0 \\
                                      1 & \,x_{31} &\,y_{31}y_{32} &\,x_{31}y_{32} &\,0 \\
                                      1 & \,x_{41}  &\,y_{41}y_{42} &\,x_{41}y_{42}  &\, 0 \\                                                                           1 & \,x_{51} &\,y_{51}y_{52} &\,x_{51}y_{52}  &\, x_{51}y_{52}x_{53}x_{54} \\
                                                                            \end{array}
                                  \right].
$$
The corresponding block diagonal form is
$$\left[
                                   \begin{array}{ccccc}
                                      1 & \,0      &\,0            &\,0  &\, 0 \\
                                      0 & \,x_{21} &\,0            &\,0  &\, 0 \\
                                      0 & \,0 &\,y_{31}y_{32} &\,x_{31}y_{32} &\,0 \\
                                      0 & \,0  &\,y_{41}y_{42} &\,x_{41}y_{42}  &\, 0 \\                                                                           0 & \,0 &\,0 &\,0  &\, x_{51}y_{52}x_{53}x_{54} \\
                                                                            \end{array}
                                  \right].
$$ \qed
\end{exmp}

Now we give an alternative definition of minimal staircase form and its partition type. It is equivalent to Definition \ref{df2:minmal staircase form}.
\begin{defn}\label{minimalst}
Suppose that $\mu=\sum m_i j_i\in\Pi_k$ is a partition of $k$, where $j_i$ are distinct positive integers.
Given a nonzero staircase form $S$, if for each $i$ the block diagonal form $B(S)$ contains exactly $m_i$ blocks with each having $j_i$ nonzero entries above the diagonal,
then we say $S$ is of \emph{partition type} $\mu$. Furthermore, if
\begin{equation}\label{diagonal_one_above}
(\text{the entry in the }i\text{-th row and }j\text{-th column in }S)= 0 \text{ for every }i,j \text{ with }j>i+1,
\end{equation}
then $S$ is called a \emph{minimal} staircase form of partition type
$\mu$. We call a block is minimal if the block satisfies condition
(\ref{diagonal_one_above}).\qed
\end{defn}

\begin{exmp}\label{eg:minimal staircase}
Let $n=11$, $k=7$, $s_1=0$, $s_2=1$, $s_3=2$, $s_4=2$, $s_5=4$, $s_6=4$, $s_7=4$, $s_8=7$, $s_9=7$, $s_{10}=8$ and $s_{11}=9$.
Then
$S$ is a staircase form of partition type $3+3+1$, but is not minimal
because there is a nonzero entry in the fifth row and seventh column. The 4-th block
is not minimal.

\begin{picture}(200,105)
\put(255,103){\line(3,-2){150}}
\put(0,50){
\tiny{
$S=\left[
                                   \begin{array}{ccccccccccc}
                                      * & 0 &0 &0  & 0 & 0& 0& 0& 0& 0& 0\\
                                      * & * &0 &0  & 0 & 0& 0& 0& 0& 0& 0\\
                                      * & * &* &*  & 0 & 0& 0& 0& 0& 0& 0\\
                                      * & * &* &*  & 0 & 0& 0& 0& 0& 0& 0\\
                                      * & * &* &*  & * & *& *& 0& 0& 0& 0\\
                                      * & * &* &*  & * & *& *& 0& 0& 0& 0\\
                                      * & * &* &*  & * & *& *& 0& 0& 0& 0\\
                                      * & * &* &*  & * & *& *& *& *& 0& 0\\
                                      * & * &* &*  & * & *& *& *& *& *& 0\\
                                      * & * &* &*  & * & *& *& *& *& *& *\\
                                      * & * &* &*  & * & *& *& *& *& *& *\\
                                   \end{array}
   \right],
\quad B(S)=\left[
                                   \begin{array}{ccccccccccc}
                                      * & 0\, &0 \, &0  & 0\, & 0& 0\,& 0\,& 0\,& 0\,& 0\\
                                      0 & * &0 &0  & 0 & 0& 0& 0& 0& 0& 0\\
                                      0 & 0 &* &*  & 0 & 0& 0& 0& 0& 0& 0\\
                                      0 & 0 &* &*  & 0 & 0& 0& 0& 0& 0& 0\\
                                      0 & 0 &0 &0  & * & *& *& 0& 0& 0& 0\\
                                      0 & 0 &0 &0  & * & *& *& 0& 0& 0& 0\\
                                      0 & 0 &0 &0  & * & *& *& 0& 0& 0& 0\\
                                      0 & 0 &0 &0  & 0 & 0& 0& *& *& 0& 0\\
                                      0 & 0 &0 &0  & 0 & 0& 0& *& *& *& 0\\
                                      0 & 0 &0 &0  & 0 & 0& 0& *& *& *& *\\
                                      0 & 0 &0 &0  & 0 & 0& 0& *& *& *& *\\
                                   \end{array}
           \right].
$}
}
\end{picture}\qed
\end{exmp}

\begin{defn}
Define a natural partial order on the set of partitions $\Pi_k$ as follows: for two partitions $\mu=(\mu_1+\cdots+\mu_s)$ and $\nu=(\nu_1+\cdots+\nu_t)$ in $\Pi_k$, define $\mu<_P\nu$ if $\mu\neq \nu$ and $\mu$ is a subpartition of $\nu$, i.e., if it is possible to rearrange the order of $\mu$ as $(\mu'_1+\cdots+\mu'_s)$ such that there exist $0=i_0<i_1< \cdots< i_{t-1}<i_t=s$ satisfying $$\mu'_{i_{r-1}+1}+\mu'_{i_{r-1}+2}+\cdots+\mu'_{i_r}=\nu_r, \quad \hbox{ for } r=1,2,\dots,t.$$
Define $\mu\le_P\nu$ if $\mu=\nu$ or $\mu<_P\nu$.
\qed
\end{defn}
\begin{exmp}
In the set $\Pi_{10}$ of partitions of $10$, we have $(4+2+2+1+1) \,
<_P \,(5+3+2)$ because we can rearrange $(4+2+2+1+1)$ to
$(4+1+2+1+2)$, and $4+1=5$, $2+1=3$, $2=2$. \qed
\end{exmp}

The following two propositions are essential ingredients to prove the main
 theorem. We will only state the propositions here but leave the proofs to \S3.

\begin{prop}\label{staircaseformofnonminimal}
Suppose $n\ge 8k+5$, $d_1,d_2\ge (2k+1)n$ and fix a partition $\mu=\sum m_i j_i\in \Pi_k$.
Then any nonzero staircase form $f$ of type $\mu$ of
bidegree $(d_1,d_2)$ is in the ideal
$$I_{<d}+(\hbox{minimal staircase forms of bidegree } (d_1,d_2)\hbox{ and of partition types} \le_P\mu),$$ that is to say,
 $f$ can be generated by elements in $I_{<d}$ and minimal staircase forms
 of the same or lower partition types
 of bidegree $(d_1,d_2)$.
\end{prop}


\begin{prop}\label{staircaseformofpartition}
Suppose $n\ge 8k+5$, $d_1,d_2\ge (2k+1)n$
and fix a partition $\mu=\sum m_i j_i\in\Pi_k$. Then any minimal staircase form of
partition type $\mu$ generates all the minimal staircase forms of the same partition
type $\mu$, modulo the ideal
$$I_{<d}+(\hbox{ minimal staircase forms of partition types $<_P\mu$}).$$
\end{prop}

\begin{cor}\label{lessthanpartitionnumber}
Suppose $n\ge 8k+5$, $d_1,d_2\ge (2k+1)n$. Then $M_{d_1,d_2}$ can be
generated by $p(k)$ elements $\{\det(S_\mu)\}_{\mu\in\Pi_k}$, where
for each partition $\mu\in\Pi_k$, $S_\mu$ is an arbitrary minimal
staircase form of bidegree $(d_1,d_2)$ and of partition type $\mu$.
In particular, $\dim M_{d_1,d_2}\le p(k)$.
\end{cor}
\begin{proof}
It is an immediate consequence of Proposition~\ref{staircaseformofnonminimal} and Proposition~\ref{staircaseformofpartition}.
\end{proof}

\section{proof of main theorem} This section is the most technical part of the paper. First we give Transfactor Lemma (Lemma \ref{lem:transfactor}) and Minors Permuting Lemma (Lemma \ref{lem:MinorsPermutingLemma}), which are simple but powerful tools to modify $D\in\mathfrak{D}$ of degree $d$ to another $D'\in\mathfrak{D}$ such that $\Delta(D)\sim\Delta(D')$ modulo $I_{<d}$. Then we prove Lemma \ref{lem:powerful} which gives a relation among the determinants of $D$ and certain modifications of $D$. After this lemma is established, we shall prove Proposition \ref{staircaseformofpartition}, Proposition \ref{staircaseformofnonminimal} and then the main theorem (Theorem \ref{mainthm}).

\begin{lem}[Transfactor Lemma]\label{lem:transfactor} Let $D=\{P_1,\dots,P_n\}\in\mathfrak{D}$ where $P_i=(\alpha_i,\beta_i)\in\N\times\N$ and define $\deg P_i:=s_i (=\alpha_i+\beta_i)$ for $1\le i\le n$ and $d=\sum s_i$. Define $s_{n+1}=n$.  Suppose $1\le i\neq j\le n$ are two integers satisfying $s_i=i-1$, $s_{i+1}=i$, $s_j=j-1$, $s_{j+1}=j$, $\beta_i>0$, $\alpha_j>0$. Define
$$D'=\{P_1, \dots, P_{i-1},
P_i+(1,-1),P_{i+1},\dots, P_{j-1},
P_j+(-1,1),P_{j+1},\dots,P_n\}.$$Then $\Delta(D)\sim\Delta(D')$ modulo $I_{<d}$.
\end{lem}

\begin{proof} By performing appropriate operations as in
Definition-Proposition \ref{definitionofstaircaseform}, we can
obtain a staircase form $S$ of $D$ (resp. staircase form $S'$ of
$D'$), such that the $(i,i)$-entry and $(j,j)$-entry of $S$ (resp.
$S'$) are $y_{i1}\prod_{t=2}^{i-1}a_{it}$ and
$x_{j1}\prod_{t=2}^{j-1}a_{jt}$ (resp.
$x_{i1}\prod_{t=2}^{i-1}a_{it}$ and
$y_{j1}\prod_{t=2}^{j-1}a_{jt}$). The block diagonal forms of $S$
and $S'$ only differ at two blocks of size 1 located at the
$(i,i)$-entry and $(j,j)$-entry. Let $f_0$ be the product of
determinants of all blocks of $B(S)$ except the $(i,i)$-entry and
$(j,j)$-entry. Then $\Delta(D)-\Delta(D')$ is equivalent to the
following modulo $I_{<d}$,
$$\aligned\det(S)-\det(S')&=\Big{(}y_{i1}\prod_{t=2}^{i-1}a_{it}\Big{)}\Big{(}x_{j1}\prod_{t=2}^{j-1}a_{jt}\Big{)}f_0
-\Big{(}x_{i1}\prod_{t=2}^{i-1}a_{it}\Big{)}\Big{(}y_{j1}\prod_{t=2}^{j-1}a_{jt}\Big{)}f_0\\
&=-\det\begin{bmatrix}1&x_1&y_1\\1&x_i&y_i\\1&x_j&y_j\end{bmatrix}\Big{(}\prod_{t=2}^{i-1}a_{it}\Big{)}\Big{(}\prod_{t=2}^{j-1}a_{jt}\Big{)}f_0.\endaligned$$
Without loss of generality, assume $i<j$. Since
$$I=\bigcap_{1\leq i<j\leq n}(x_i-x_j, y_i-y_j),$$it is easy to see
that $(\det(S)-\det(S'))/a_{ji}$ is a polynomial in $I_{<d}$ and then the lemma follows.
\end{proof}

The Transfactor Lemma immediately leads to the proof of the following lemma, which is the base case $k=0$ of the inductive proof of Proposition \ref{staircaseformofpartition}.

\begin{lem}\label{lem:base case}
Let $d_1, d_2$ be two non-negative integers such that
$d_1+d_2=n(n-1)/2$. Then $M_{d_1, d_2}$ is generated by any single
nonzero staircase form modulo $I_{<n(n-1)/2}$.
\end{lem}
\begin{proof}
Let $S$ be a staircase form with $\det S\neq 0$ and bidegree $(d_1,d_2)$. Because $d_1+d_2=n(n-1)/2$, there are $n(n-1)/2$  zeros in the staircase form $S$. Since $\det S\neq 0$, $S$ and its block diagonal form $B(S)$ must be of the following form
$$S=\tiny{\left[
                                   \begin{array}{ccccc}
                                      * & 0 &\cdots &0 &0\\
                                      * & * &\cdots &0 &0\\
                                      \vdots & \vdots &\ddots &\vdots&\vdots\\
                                      * & * &\cdots &* &0\\
                                      * & * &\cdots &* &*\\
                                                                 \end{array}
                                  \right],
                                 }\quad\normalsize{\text{ $B(S)$=}}
                                  \tiny{\left[
                                   \begin{array}{ccccc}
                                      * & 0 &\cdots &0 &0\\
                                      0& * &\cdots &0 &0\\
                                      \vdots & \vdots &\ddots &\vdots&\vdots\\
                                      0& 0 &\cdots &* &0\\
                                      0 & 0 &\cdots &0 &*\\
                                                                 \end{array}
                                  \right].}
                                  $$
By repeatedly applying Transfactor Lemma we can easily deduce the following assertion: if $S'$ is a staircase form of another $D'$ of the same bidegree ($d_1,d_2$) as $D$, then $\det B(S')\sim\det B(S)$ modulo $I_{<n(n-1)/2}$. The lemma follows from this assertion.
\end{proof}

\begin{lem}[Minors Permuting Lemma]\label{lem:MinorsPermutingLemma}Let $D=\{P_1,\dots,P_n\}\in\mathfrak{D}$ where $P_i=(\alpha_i,\beta_i)$. Suppose $h, \ell$ and $m$ are
positive integers satisfying
$2\le h<h+\ell+m\le n+1$,  $s_h=h-1,s_{h+\ell}=h+\ell-1$, $s_{h+\ell+m}=h+\ell+m-1$
(this condition holds if
$h+\ell+m=n+1$ since we assume $s_{n+1}=n$) and suppose that
$\alpha_{h+\ell},...,\alpha_{h+\ell+m-1}\geq \ell$. Define
$$\aligned D'=\{&P_1,P_2, \dots, P_{h-1},
P_{h+\ell}-(\ell,0),P_{h+\ell+1}-(\ell,0),\dots,P_{h+\ell+m-1}-(\ell,0),\\
&P_{h}+(m,0),P_{h+1}+(m,0),\dots,P_{h+\ell-1}+(m,0),
P_{h+\ell+m},\dots,P_{n}\}.\endaligned$$
Then $\Delta(D)\sim\Delta(D')$  modulo $I_{<d}$.
\end{lem}
\begin{proof}
By performing appropriate operations as in
Definition-Proposition~\ref{definitionofstaircaseform} and using the
assumption that $\alpha_v\ge l$ for $h+\ell\le v\le h+\ell+m-1$ , we
can obtain a staircase form $S$ of $D$,  where the $(u,v)$-entry for
$ h+\ell\le u, v\le h+\ell+m-1$  contains the factor
$\prod_{j=h}^{h+\ell-1} x_{uj}=\prod_{j=h}^{h+\ell-1} (x_u-x_j)$.
Let $B(S)=\hbox{diag}(B_1,B_2,\dots,B_s)$ be the block diagonal form
of $S$, and let $B_r$ (resp. $B_{r+1}$) be the block of size $\ell$
(resp. $m$) whose upper left corner is the $(h,h)$-entry (resp.
$(h+\ell,h+\ell)$-entry). Then by our choice of $S$, all entries in
the $i$-th row ($1\le i\le m$) of $B_{r+1}$ contain
$\prod_{j=h}^{h+\ell-1} x_{i+h+\ell-1,j}$ as  a factor. Dividing the
$i$-th row of $B_{r+1}$ by $\prod_{j=h}^{h+\ell-1} x_{i+h+\ell-1,j}$
for $1\le i\le m$ and multiplying the $i'$-th row of $B_r$ by
$\prod_{j=h+\ell}^{h+\ell+m-1} x_{i'+h-1,j}$ for $1\le i'\le \ell$,
we obtain a new block diagonal matrix
$B'=\hbox{diag}(B_1,\dots,B_{r-1}, B'_r,B'_{r+1},B_{r+2},\dots,
B_s)$. Since
$$\prod_{j=h}^{h+\ell-1} x_{i+h+\ell-1,j}=(-1)^{\ell m}\prod_{j=h}^{h+\ell-1} x_{i+h+\ell-1,j},$$
$$(-1)^{\ell m}\det B'=\det B=\det S.$$
Now interchange the two blocks $B'_r$ and $B'_{r+1}$ in $B'$ and
then change the indices $1,\dots,n$ to
$$1,\dots,(\ell-1),(\ell+h),\dots,(\ell+h+m-1),\ell,\dots,(\ell+h-1),
(\ell+h+m),\dots,n.$$ The resulting matrix is the block diagonal
matrix of a staircase form of $D'$. Notice that when we change the
indices, the determinant of the resulting matrix is equal to $\det
B'$ multiplied by $(-1)^{\ell m}$. Therefore $\Delta(D)$ and
$\Delta(D')$ are equivalent modulo $I_{<d}$.
\end{proof}

\begin{exmp} In Example \ref{eg:minimal staircase},
assume $\alpha_8,\dots,\alpha_{11}\ge 3$. Lemma
\ref{lem:MinorsPermutingLemma} asserts that by permuting the two
blocks (as framed in the following figure) in the block diagonal
form, the determinant is not changed modulo $I_{<d}$. That is to
say, we may permute adjacent blocks provided that the $\alpha_i$'s
in the second block is not less than the size of the first block.

\begin{picture}(100,110)
\put(71,38){\framebox(38,29){}}
\put(114,-1){\framebox(52,38){}}
\put(332,-1){\framebox(38,29){}}
\put(275,29){\framebox(52,38){}}
\put(0,50){
\tiny{
$\left[
                                   \begin{array}{ccccccccccc}
                                      * & 0      &0            &0  & 0 & 0& 0& 0& 0& 0& 0\\
                                      0 & * &0            &0  & 0 & 0& 0&  0& 0& 0& 0\\
                                      0 & 0 &* &* &0 & 0& 0& 0& 0& 0& 0\\
                                      0 & 0 &* &*  & 0 & 0& 0& 0& 0& 0& 0\\
                                                             0 & 0 &0 &0  & * & *& *& 0& 0& 0& 0\\
                                                             0 & 0 &0 &0  & * & *& *& 0& 0& 0& 0\\
                                                             0 & 0 &0 &0  & * & *& *& 0& 0& 0& 0\\
                                                             0 & 0 &0 &0  & 0 & 0& 0& *& *& 0& 0\\
                                                             0 & 0 &0 &0  & 0 & 0& 0& *& *& *& 0\\
                                                             0 & 0 &0 &0  & 0 & 0& 0& *& *& *& *\\
                                                             0 & 0 &0 &0  & 0 & 0& 0& *& *& *& *\\
                                                                            \end{array}
                                  \right]
\quad \longrightarrow\quad \left[
                                   \begin{array}{ccccccccccc}
                                      * & 0      &0            &0  & 0 & 0& 0& 0& 0& 0& 0\\
                                      0 & * &0            &0  & 0 & 0& 0&  0& 0& 0& 0\\
                                      0 & 0 &* &* &0 & 0& 0& 0& 0& 0& 0\\
                                      0 & 0 &* &*  & 0 & 0& 0& 0& 0& 0& 0\\
                                                             0 & 0 &0 &0  & * & *& 0& 0& 0& 0& 0\\
                                                             0 & 0 &0 &0  & * & *& *& 0& 0& 0& 0\\
                                                             0 & 0 &0 &0  & * & *& *& *& 0& 0& 0\\
                                                             0 & 0 &0 &0  & * & *& *& *& 0& 0& 0\\
                                                             0 & 0 &0 &0  & 0 & 0& 0& 0& *& *& *\\
                                                             0 & 0 &0 &0  & 0 & 0& 0& 0& *& *& *\\
                                                             0 & 0 &0 &0  & 0 & 0& 0& 0& *& *& *\\
                                                                            \end{array}
                                  \right]
$}
}
\end{picture}
\end{exmp}

We frequently use the following elementary lemma.
\begin{lem}\label{sum} For any non-negative integers $c$ and $e$,
$$
(x_1^c y_1^e + \cdots +x_n^c y_n^e)\cdot\emph{det}\left[
                                   \begin{array}{cccc}
                                      x_{1}^{\alpha_{1}} y_{1}^{\beta_{1}} & \,x_{1}^{\alpha_{2}} y_{1}^{\beta_{2}} &\cdots  &\, x_{1}^{\alpha_{n}} y_{1}^{\beta_{n}} \\
                                      x_{2}^{\alpha_{1}} y_{2}^{\beta_{1}} & \,x_{2}^{\alpha_{2}} y_{2}^{\beta_{2}} &\cdots &\, x_{2}^{\alpha_{n}} y_{2}^{\beta_{n}} \\
                                      \vdots & \vdots & \ddots &\vdots \\
                                      x_{n}^{\alpha_{1}} y_{n}^{\beta_{1}} & \,x_{n}^{\alpha_{2}} y_{n}^{\beta_{2}} &\cdots &\, x_{n}^{\alpha_{n}} y_{n}^{\beta_{n}} \\
                                                                            \end{array}
                                  \right]
$$

$$
=\emph{det}\left[
                                   \begin{array}{cccc}
                                      x_{1}^{\alpha_{1}+c} y_{1}^{\beta_{1}+e} & \,x_{1}^{\alpha_{2}} y_{1}^{\beta_{2}} &\cdots  &\, x_{1}^{\alpha_{n}} y_{1}^{\beta_{n}} \\
                                      x_{2}^{\alpha_{1}+c} y_{2}^{\beta_{1}+e} & \,x_{2}^{\alpha_{2}} y_{2}^{\beta_{2}} &\cdots &\, x_{2}^{\alpha_{n}} y_{2}^{\beta_{n}} \\
                                      \vdots & \vdots & \ddots &\vdots \\
                                      x_{n}^{\alpha_{1}+c} y_{n}^{\beta_{1}+e} & \,x_{n}^{\alpha_{2}} y_{n}^{\beta_{2}} &\cdots &\, x_{n}^{\alpha_{n}} y_{n}^{\beta_{n}} \\
                                                                            \end{array}
                                  \right]
                                  +\cdots+\emph{det}\left[
                                   \begin{array}{cccc}
                                      x_{1}^{\alpha_{1}} y_{1}^{\beta_{1}} & \,x_{1}^{\alpha_{2}} y_{1}^{\beta_{2}} &\cdots  &\, x_{1}^{\alpha_{n}+c} y_{1}^{\beta_{n}+e} \\
                                      x_{2}^{\alpha_{1}} y_{2}^{\beta_{1}} & \,x_{2}^{\alpha_{2}} y_{2}^{\beta_{2}} &\cdots &\, x_{2}^{\alpha_{n}+c} y_{2}^{\beta_{n}+e} \\
                                      \vdots & \vdots & \ddots &\vdots \\
                                      x_{n}^{\alpha_{1}} y_{n}^{\beta_{1}} & \,x_{n}^{\alpha_{2}} y_{n}^{\beta_{2}} &\cdots &\, x_{n}^{\alpha_{n}+c} y_{n}^{\beta_{n}+e} \\
                                                                            \end{array}
                                  \right].
$$
\end{lem}

\begin{proof}
The Lemma is the special case of the following Lemma where
$$p=x_1^cy_1^e,\quad
q=x_1^{\alpha_1}y_1^{\beta_1}x_2^{\alpha_2}y_2^{\beta_2}\cdots
x_n^{\alpha_n}y_n^{\beta_n}.$$
\end{proof}

\begin{lem}\label{lem1}
For $p,q\in \mathbb{C}[\mathbf{x},\mathbf{y}]$, we have
$$A(Sym(p)q)=Sym(p)A(q)$$ where $Sym(p)$ denotes the symmetric sum
$\sum_{\sigma\in S_n}\sigma(p)$ and $A(p)$ denotes the alternating
sum $\sum_{\sigma\in S_n}\emph{sign}(\sigma)\sigma(p)$.
\end{lem}
\begin{proof}
 Since the polynomial $Sym(p)$ is invariant under $S_n$ action, $$A(Sym(p)q)=\sum_\sigma \mbox{sign}(\sigma) \sigma(Sym(p)q)=\sum_\sigma\mbox{sign}(\sigma) Sym(p)\sigma(q)=Sym(p)A(q).$$
\end{proof}
 \begin{remark}
 Lemma~\ref{sum} implies that the determinants on the right hand side of the equality are linearly dependent modulo $I_{<d}$ where $d=\sum_{i=1}^n(\alpha_i+\beta_i)+c+e$. It is easier to express the dependency in terms of squares in $\N\times \N$: for $D=\{(\alpha_1,\beta_1),\dots,(\alpha_n,\beta_n)\}\in\mathfrak{D}$, let $D_i\in\mathfrak{D}$ be obtained from $D$ by replacing $(\alpha_i,\beta_i)$ by $(\alpha_i+c,\beta_i+e)$. Then Lemma~\ref{sum} asserts that
$$\Delta(D_1)+\Delta(D_2)+\cdots+\Delta(D_n)\sim 0 \quad \hbox{ modulo } I_{<d}.$$
Up to modulo $I_{<d}$, we can replace $\Delta(D_i)$ by a linear combination of $\Delta(D_j)$ for $j\neq i$. To say it more vividly, $D_j$ is obtained from $D_i$ by sending $(\alpha_i+c,\beta_i+e)$ to $(\alpha_i,\beta_i)$, and then sending $(\alpha_j,\beta_j)$ to $(\alpha_j+c,\beta_j+e)$. \qed
\end{remark}

\begin{lem}\label{lem:powerful}
 Let $D=\{P_1,\dots,P_n\}\in\mathfrak{D}$ where $P_i=(\alpha_i,\beta_i)$ are not necessarily distinct and $\{s_i:=\alpha_i+\beta_i\}_{1\le i\le n}$ are weakly increasing. Let $S$ be a staircase form of $D$ and $B(S)$ its block diagonal form.  Suppose the last block of $B(S)$ is of size $t_0$ and in this block there are $j_r$ nonzero entries above the diagonal. Suppose the first $(j_r+2)$ blocks of $B(S)$ are of size $1$, i.e., $s_i=i-1$ for $1\le i\le j_r+3$. Suppose $P_2=(1,0)$. Let $t$ be an integer that $1\le t\le t_0$. Suppose $\alpha_{n-t+1},\beta_{n-t+1}\ge 1$. Let
$$\aligned &D^\nwarrow=\{P_1,\dots,P_{j_r+1},P_{j_r+2}+(1,-1),P_{j_r+3},\dots,P_{n-t},P_{n-t+1}+(-1,1),P_{n-t+2},\dots,P_n\},\\
&D^\searrow=\{P_1,(0,1),P_3,\dots,P_{n-t},P_{n-t+1}+(1,-1),P_{n-t+2},\dots,P_n\}.\endaligned$$
Then $2\Delta(D)\sim\Delta(D^\nwarrow)+\Delta(D^\searrow)$ modulo
$I_{<d}$ and staircase forms of lower partition types. Moreover, if
the last block of $B(S)$ is not minimal or if $s_{n-t+1}>n-t_0$,
then $\Delta(D)\sim\Delta(D^\searrow)$ modulo $I_{<d}$ and staircase
forms of lower partition types.
\end{lem}
\begin{proof}
The reader is strongly recommended to see Example~\ref{criticalexample} first.

Suppose the partition type of $D$ is $$\underbrace{j_1+\cdots+j_1}_{m_1}+\cdots+\underbrace{j_{r-1}+\cdots+j_{r-1}}_{m_{r-1}}+
\underbrace{j_r+\cdots+j_r}_{m_r}.$$
Applying Lemma \ref{sum} to
$$(\sum x_i^{\alpha_{n-t+1}}y_i^{\beta_{n-t+1}-1})\cdot \Delta(\{P_1,(0,1),P_2,\dots,\widehat{P}_{n-t+1},\dots,P_n\}),$$
 which is an element in $I_{<d}$, we get a sum of $n$ determinants:
 the 1st determinant is in $I_{<d}$ because the first row of its staircase form is the zero row.
  The 2nd determinant is
\begin{equation}\label{eq:1}\Delta(\{P_1,P_{n-t+1},P_2,\dots,\widehat{P}_{n-t+1},\dots,P_n\})=(-1)^{n-t-1}\Delta(D).\end{equation}
The $i$-th determinant for $i\ge 3$ is
$$\Delta(\{P_1,(0,1),P_2,\dots,P_{i-2}, P_{i-1}+P_{n-t+1}-(0,1), P_i,\dots,\widehat{P}_{n-t+1},\dots,P_n\}),$$
when $3\le i\le j_r+3$, its partition type is equal to or lower than
$$\underbrace{j_1+\cdots+j_1}_{m_1}+\cdots+\underbrace{j_{r-1}+\cdots+j_{r-1}}_{m_{r-1}}+
    \underbrace{j_r+\cdots+j_r}_{m_r-1}+(i-3)+(j_r-i+3)$$
which is strictly lower than the partition type of $D$ when $4\le
i\le j_r+2$;  when $i>j_r+3$, the determinant is equivalent to 0. So
modulo $I_{<d}$ and staircase forms of lower partition types, the
sum of
\begin{equation}\label{eq:2}\Delta(\{P_1,(0,1), P_2+P_{n-t+1}-(0,1), P_3,\dots,\widehat{P}_{n-t+1},\dots,P_n\}),
\end{equation}
\begin{equation}\label{eq:3}\Delta(\{P_1,(0,1), P_2,\dots,P_{j_r+1}, P_{j_r+2}+P_{n-t+1}-(0,1), P_{j_r+3},\dots,\widehat{P}_{n-t+1},\dots,P_n\})
\end{equation}
and (\ref{eq:1}) is equivalent to 0.

Similarly as above, applying Lemma \ref{sum} to $$(\sum
x_i^{\alpha_{n-t+1}-1}y_i^{\beta_{n-t+1}})\cdot
\Delta(\{P_1,(0,1),P_2,\dots,P_{j_r+1},P_{j_r+2}+(1,-1),P_{j_r+3},\dots,\widehat{P}_{n-t+1},\dots,P_n\}),$$
which is an element in $I_{<d}$, we get a sum of $n$ determinants:
the 2nd determinant is
\begin{equation}\label{eq:4}\Delta(\{P_1,P_{n-t+1}+(-1,1),P_2,\dots,P_{j_r+1},P_{j_r+2}+(1,-1),P_{j_r+3},\dots,\widehat{P}_{n-t+1},\dots,P_n\}),\end{equation}
the 3rd determinant is
\begin{equation}\label{eq:5}\Delta(\{P_1,(0,1),P_{n-t+1},P_3,\dots,P_{j_r+1},P_{j_r+2}+(1,-1),P_{j_r+3},\dots,\widehat{P}_{n-t+1},\dots,P_n\}),
\end{equation}
the ($j_r+3$)-th determinant is
\begin{equation}\label{eq:6}\Delta(\{P_1,(0,1),P_2,\dots,P_{j_r+1},P_{j_r+2}+P_{n-t+1}-(0,1),P_{j_r+3},\dots,\widehat{P}_{n-t+1},\dots,P_n\}),
\end{equation}
and all other determinants are equivalent to 0 modulo $I_{<d}$ and
staircase forms of partition types lower than $D$. Now compare the
two relations we obtained:
$$\left\{
\begin{array}{l} (\ref{eq:1})+(\ref{eq:2})+(\ref{eq:3})\sim 0,\\
          (\ref{eq:4})+(\ref{eq:5})+(\ref{eq:6})\sim 0.
\end{array}
\right.$$
Note that by Transfactor Lemma (Lemma \ref{lem:transfactor}), the polynomial (\ref{eq:5}) is equivalent to
$$\aligned&\Delta(\{P_1,P_2,P_{n-t+1},P_3,\dots,P_{j_r+1},P_{j_r+2},P_{j_r+3},\dots,\widehat{P}_{n-t+1},\dots,P_n\})\\
&=(-1)^{n-t-2}\Delta(D)=-(\ref{eq:1}),
\endaligned$$
and also note that (\ref{eq:3})=(\ref{eq:6}). So we have
$$(\ref{eq:4}) \sim -(\ref{eq:5})-(\ref{eq:6}) \sim  (\ref{eq:1})-(\ref{eq:3}) \sim  2(\ref{eq:1})+(\ref{eq:2}).$$ Since (\ref{eq:4})=$(-1)^{n-t-1}\Delta(D^\nwarrow)$ and (\ref{eq:2})$=(-1)^{n-t-2}\Delta(D^\searrow)$, the lemma follows.

Note that since $\deg P_{n-t+1}\ge \deg P_{n-t_0+1}=n-t_0$, we have  $$\deg \big{(}P_{j_r+2}+P_{n-t+1}-(0,1)\big{)}\ge (j_r+1)+(n-t_0)-1=j_r+n-t_0$$ which is greater than $n-1$ if $j_r\ge t_0$. But this is always the case if the last block of $B(S)$ is not minimal. In this case,  (\ref{eq:3})$\sim 0$ and therefore $(\ref{eq:1})+(\ref{eq:2})\sim 0$. Of course we still have $(\ref{eq:1})+(\ref{eq:2})\sim 0$ if $(s_{n-t+1}=)\deg P_{n-t+1}>n-t_0$.
\end{proof}

The discovery of Lemma \ref{lem:powerful} is motivated by the observation in the following example.
\begin{exmp}\label{criticalexample}
Let $n=9$, $k=3$, then $d={9 \choose 2}-3=33$. Consider a partition $(2+1)\in\Pi_3$, and let $t=3$. Let $D=$
\begin{picture}(85,50)
\put(0,0){\line(0,1){40}} \put(0,0){\line(1,0){80}}
\put(10,0){\line(0,1){40}} \put(0,10){\line(1,0){80}}
\put(20,0){\line(0,1){40}} \put(0,20){\line(1,0){80}}
\put(30,0){\line(0,1){40}} \put(0,30){\line(1,0){80}}
\put(40,0){\line(0,1){40}} \put(0,40){\line(1,0){80}}
\put(50,0){\line(0,1){40}}
\put(60,0){\line(0,1){40}}
\put(70,0){\line(0,1){40}}
\put(80,0){\line(0,1){40}}
\put(-3,-3){$\bullet$}\put(7,-3){$\bullet$}
\put(-3,17){$\bullet$}\put(-3,27){$\bullet$}
\put(17,17){$\bullet$}\put(27,7){$\bullet$}
\put(47,7){$\bullet$}\put(57,-3){$\bullet$}
\put(67,-3){$\bullet$}
\end{picture} (in the standard order) and $f=\Delta(D)$.

\noindent (i) Applying Lemma \ref{sum} to the product of  $\sum_{i=1}^9 x_i^5 y_i^0$ with $\Delta$(
\begin{picture}(80,50)
\put(0,0){\line(0,1){40}} \put(0,0){\line(1,0){80}}
\put(10,0){\line(0,1){40}} \put(0,10){\line(1,0){80}}
\put(20,0){\line(0,1){40}} \put(0,20){\line(1,0){80}}
\put(30,0){\line(0,1){40}} \put(0,30){\line(1,0){80}}
\put(40,0){\line(0,1){40}} \put(0,40){\line(1,0){80}}
\put(50,0){\line(0,1){40}}
\put(60,0){\line(0,1){40}}
\put(70,0){\line(0,1){40}}
\put(80,0){\line(0,1){40}}
\put(-3,-3){$\bullet$}\put(7,-3){$\bullet$}
\put(-3,17){$\bullet$}\put(-3,27){$\bullet$}
\put(17,17){$\bullet$}\put(27,7){$\bullet$}
\put(-3,7){$\bullet$}\put(57,-3){$\bullet$}
\put(67,-3){$\bullet$}
\end{picture}
). Modulo $I_{<d}+($minimal staircase forms of lower partitions), there are 2 summands
remained in the sum: $\Delta$(
\begin{picture}(84,50)
\put(0,0){\line(0,1){40}} \put(0,0){\line(1,0){80}}
\put(10,0){\line(0,1){40}} \put(0,10){\line(1,0){80}}
\put(20,0){\line(0,1){40}} \put(0,20){\line(1,0){80}}
\put(30,0){\line(0,1){40}} \put(0,30){\line(1,0){80}}
\put(40,0){\line(0,1){40}} \put(0,40){\line(1,0){80}}
\put(50,0){\line(0,1){40}}
\put(60,0){\line(0,1){40}}
\put(70,0){\line(0,1){40}}
\put(80,0){\line(0,1){40}}
\put(-3,-3){$\bullet$}\put(7,-3){$\bullet$}
\put(-3,17){$\bullet$}\put(-3,27){$\bullet$}
\put(17,17){$\bullet$}\put(27,7){$\bullet$}
\put(47,7){$\bullet$}\put(57,-3){$\bullet$}
\put(67,-3){$\bullet$}\put(-3,7){$\circ$}
\put(35,8.2){$\rightarrow$}
\put(37,11){\line(-1,0){35}}
\end{picture})
and $\Delta$(
\begin{picture}(84,50)
\put(0,0){\line(0,1){40}} \put(0,0){\line(1,0){80}}
\put(10,0){\line(0,1){40}} \put(0,10){\line(1,0){80}}
\put(20,0){\line(0,1){40}} \put(0,20){\line(1,0){80}}
\put(30,0){\line(0,1){40}} \put(0,30){\line(1,0){80}}
\put(40,0){\line(0,1){40}} \put(0,40){\line(1,0){80}}
\put(50,0){\line(0,1){40}}
\put(60,0){\line(0,1){40}}
\put(70,0){\line(0,1){40}}
\put(80,0){\line(0,1){40}}
\put(-3,-3){$\bullet$}\put(7,-3){$\bullet$}
\put(-3,17){$\bullet$}\put(47,27){$\bullet$}
\put(17,17){$\bullet$}\put(27,7){$\bullet$}
\put(-3,7){$\bullet$}\put(57,-3){$\bullet$}
\put(67,-3){$\bullet$}\put(-3,27){$\circ$}
\put(35,28.2){$\rightarrow$}
\put(37,31){\line(-1,0){35}}
\end{picture}).  Since the former is $\pm f$, the latter is
equivalent to $\pm f$ modulo $I_{<d}+($minimal staircase forms of lower partitions).

\noindent (ii) On the other hand, by Transfactor Lemma, $\Delta$(
\begin{picture}(84,50)
\put(0,0){\line(0,1){40}} \put(0,0){\line(1,0){80}}
\put(10,0){\line(0,1){40}} \put(0,10){\line(1,0){80}}
\put(20,0){\line(0,1){40}} \put(0,20){\line(1,0){80}}
\put(30,0){\line(0,1){40}} \put(0,30){\line(1,0){80}}
\put(40,0){\line(0,1){40}} \put(0,40){\line(1,0){80}}
\put(50,0){\line(0,1){40}}
\put(60,0){\line(0,1){40}}
\put(70,0){\line(0,1){40}}
\put(80,0){\line(0,1){40}}
\put(-3,-3){$\bullet$}\put(-3,7){$\bullet$}
\put(7,17){$\bullet$}\put(-3,27){$\circ$}
\put(17,17){$\bullet$}\put(27,7){$\bullet$}
\put(47,7){$\bullet$}\put(57,-3){$\bullet$}
\put(67,-3){$\bullet$}
\put(7,-3){$\circ$}\put(-3,17){$\bullet$}
\end{picture}) is in $I_{<d}+($minimal staircase forms of lower
partitions)+$(f)$. Note that we move the points $(1,0)$ and $(0,3)$
in $D$.

\noindent (iii) Applying Lemma \ref{sum} to the product of  $\sum_{i=1}^9 x_i^4 y_i^1$ with $\Delta$(
\begin{picture}(82,50)
\put(0,0){\line(0,1){40}} \put(0,0){\line(1,0){80}}
\put(10,0){\line(0,1){40}} \put(0,10){\line(1,0){80}}
\put(20,0){\line(0,1){40}} \put(0,20){\line(1,0){80}}
\put(30,0){\line(0,1){40}} \put(0,30){\line(1,0){80}}
\put(40,0){\line(0,1){40}} \put(0,40){\line(1,0){80}}
\put(50,0){\line(0,1){40}}
\put(60,0){\line(0,1){40}}
\put(70,0){\line(0,1){40}}
\put(80,0){\line(0,1){40}}
\put(-3,-3){$\bullet$}\put(-3,7){$\bullet$}
\put(7,17){$\bullet$}\put(-3,17){$\bullet$}
\put(17,17){$\bullet$}\put(27,7){$\bullet$}
\put(57,-3){$\bullet$}
\put(67,-3){$\bullet$}\put(7,-3){$\bullet$}
\end{picture}).
Modulo $I_{<d}+($minimal staircase forms of lower partitions), we have 3 summands left:
$\Delta$(
\begin{picture}(84,50)
\put(0,0){\line(0,1){40}} \put(0,0){\line(1,0){80}}
\put(10,0){\line(0,1){40}} \put(0,10){\line(1,0){80}}
\put(20,0){\line(0,1){40}} \put(0,20){\line(1,0){80}}
\put(30,0){\line(0,1){40}} \put(0,30){\line(1,0){80}}
\put(40,0){\line(0,1){40}} \put(0,40){\line(1,0){80}}
\put(50,0){\line(0,1){40}}
\put(60,0){\line(0,1){40}}
\put(70,0){\line(0,1){40}}
\put(80,0){\line(0,1){40}}
\put(-3,-3){$\bullet$}\put(-3,7){$\bullet$}
\put(7,17){$\bullet$}\put(-3,17){$\bullet$}
\put(17,17){$\bullet$}\put(27,7){$\bullet$}
\put(47,7){$\bullet$}\put(57,-3){$\bullet$}
\put(67,-3){$\bullet$}\put(7,-3){$\circ$}\put(48,10){\line(-4,-1){36}}\put(42,8){\tiny{$>$}}
\end{picture}),
$\Delta$(
\begin{picture}(84,50)
\put(0,0){\line(0,1){40}} \put(0,0){\line(1,0){80}}
\put(10,0){\line(0,1){40}} \put(0,10){\line(1,0){80}}
\put(20,0){\line(0,1){40}} \put(0,20){\line(1,0){80}}
\put(30,0){\line(0,1){40}} \put(0,30){\line(1,0){80}}
\put(40,0){\line(0,1){40}} \put(0,40){\line(1,0){80}}
\put(50,0){\line(0,1){40}}
\put(60,0){\line(0,1){40}}
\put(70,0){\line(0,1){40}}
\put(80,0){\line(0,1){40}}
\put(-3,-3){$\bullet$}\put(-3,7){$\bullet$}
\put(-3,17){$\bullet$}\put(7,17){$\circ$}
\put(17,17){$\bullet$}\put(27,7){$\bullet$}
\put(47,27){$\bullet$}\put(57,-3){$\bullet$}
\put(67,-3){$\bullet$}\put(7,-3){$\bullet$}\put(48,30){\line(-4,-1){36}}\put(42,28){\tiny{$>$}}
\end{picture}),
and $\Delta$(
\begin{picture}(84,50)
\put(0,0){\line(0,1){40}} \put(0,0){\line(1,0){80}}
\put(10,0){\line(0,1){40}} \put(0,10){\line(1,0){80}}
\put(20,0){\line(0,1){40}} \put(0,20){\line(1,0){80}}
\put(30,0){\line(0,1){40}} \put(0,30){\line(1,0){80}}
\put(40,0){\line(0,1){40}} \put(0,40){\line(1,0){80}}
\put(50,0){\line(0,1){40}}
\put(60,0){\line(0,1){40}}
\put(70,0){\line(0,1){40}}
\put(80,0){\line(0,1){40}}
\put(-3,-3){$\bullet$}\put(-3,7){$\circ$}
\put(-3,17){$\bullet$}\put(7,17){$\bullet$}
\put(17,17){$\bullet$}\put(27,7){$\bullet$}
\put(37,17){$\bullet$}\put(57,-3){$\bullet$}
\put(67,-3){$\bullet$}\put(7,-3){$\bullet$}\put(38,20){\line(-4,-1){36}}\put(32,18){\tiny{$>$}}
\end{picture}). We already know that the first two are
in the ideal $I_{<d}+($minimal staircase forms of lower partitions)+$(f)$,
hence the last one as well. \qed
\end{exmp}

\begin{proof}[Proof of Proposition~\ref{staircaseformofpartition}] First
we explain the condition $n\ge 8k+5$. It follows from the conditions
$d_1+d_2\le n(n-1)/2$ and $d_1,d_2\ge (2k+1)n$, which imply
$n(n-1)/2\ge 2(2k+1)n$, equivalently $n\ge 8k+5$.

We prove by induction on $k$. The base case $k=0$ is proved in Lemma \ref{lem:base case}. Suppose the proposition is proved for $<k$.

Let $D=\{P_1,\dots,P_n\}\in\mathfrak{D}$, $S$ be a minimal staircase
form of $D$ of partition type $\mu$. Notice that, without loss of
generality, we can assume that the last block of $B(S)$ is of size
greater than $1$. Indeed, suppose the last block, which corresponds
to $P_n$, is of size 1, and suppose that the block $M$ is the last
block among those of size greater than 1. Since $d_1\ge (2k+1)n$,
there are sufficient size-1 blocks in $B(S)$, such that by
successively moving a $P_i$ corresponding to a size-1 block to
northwest direction and moving $P_n$ to southeast direction using
Transfactor Lemma, we can assume $P_n=(\alpha_n,0)$. (Of course
$d_1\ge (2k+1)n$ is not a sharp bound. We obtain this bound by
noticing that there are at most $2k$ points of $D$ that do not
correspond to size-1 blocks, the $x$-degree of each of which is less
than $n$, while the last point $P_n$ also has $x$-degree less than
$n$. So as long as the total $x$-degree is larger than $2k\cdot
n+n$, the point $P_n$ can be moved to southeast direction by
Transfactor lemma.) Then we can apply Minors Permuting Lemma to
permute the last block with the blocks before it until it moves in
front of $M$. Then $M$ is moved to the lower right in a block
diagonal form. This procedure can be repeated until $M$ becomes the
last block.

By Transfactor Lemma and Minors Permuting Lemma together with the condition that $n\ge 8k+5$, we can assume the first $(k+2)$ blocks of $B(S)$ are all of size 1.

Now we are in the position to apply Lemma \ref{lem:powerful}.
Denote by $t_0$ the size of the last block in $B(S)$. By Transfactor Lemma we may
assume $P_2=(1,0)$. If for $1\le t\le t_0$ the point $P_{n-t+1}$ has degree $s_{n-t+1}>n-t_0$,
 then $D\sim D^\searrow$, which means that we can move
$P_{n-t+1}$ to $P_{n-t+1}+(1,-1)$. Successively applying this
procedure, we may assume that all points $P_i$ for $i>n-t+2$ have
y-coordinates 0.

Define $a(D)=\alpha_{n-t_0+2}-\alpha_{n-t_0+1}$. Then
$$a(D^\nwarrow)-1=a(D)=a(D^\searrow)+1.$$ Consider the special case
when $P_{n-t_0+1}=P_{n-t_0+2}$. In this case $\Delta(D)=0$ hence
$\Delta(D^\nwarrow)\sim-\Delta(D^\searrow)$, $a(D^\nwarrow)=1$ and
$a(D^\searrow)=-1$. Let $D''$ be the set obtained by interchanging
the $(n-t_0+1)$-th and $(n-t_0+2)$-th points in $D^\searrow$. Now we
compare $D^{\nwarrow}=\{P'_1,\dots,P'_n\}$ with
$D''=\{P''_1,\dots,P''_n\}$:

\begin{itemize}
\item They both give minimal staircase forms with the same partition type as $S$,

\item $a(D^\nwarrow)=a(D'')=1$,

\item $\Delta(D^\nwarrow)\sim \Delta(D'')$,

\item $P''_i=
\left\{
\begin{array}{ll}P'_i+(1,-1), &\hbox{ for } i=n-t+1, n-t+2;\\
         P'_i+(-1,1), &\hbox{ for } i=2,j_r+2;\\
         P'_i, &\hbox{ otherwise}.
\end{array}
\right.$

\end{itemize}
In other words, we can move $P'_{n-t+1}$ and $P'_{n-t+2}$ of
$D^\nwarrow$ to southeast direction and move two size-1 blocks of
$D^\nwarrow$ to northwest direction simultaneously without changing
$\Delta(D^\nwarrow)$ modulo the equivalence relation. Repeat the
procedure until the y-coordinates of the $(n-t+1)$-th and
$(n-t+2)$-th points are 1 and 0, respectively. Then apply the
inductive assumption for the first $n-t$ points, we can draw the
following conclusion:

For any $D'$ and $D''$ such that
\begin{itemize}
\item[(i)] both have minimal staircase forms,
\item[(ii)] their staircase forms are of the same partition type,
\item[(iii)] $\Delta(D)$ and $\Delta(D')$ have the same bi-degree,
\item[(iv)] $a(D')=a(D'')=1$,
\end{itemize}
then $\Delta(D')\sim \pm\Delta(D'')$. If (ii) is replaced by a stronger condition:
\begin{itemize}
\item[(ii)$'$] they are both in standard order and their block diagonal forms are of the same shape (i.e. for any $i$, the size of the $i$-th blocks in both block diagonal forms are the same),
\end{itemize}

\noindent then $\Delta(D')\sim \Delta(D'')$.

By Lemma \ref{lem:powerful}, we can also show that, under condition (i) (ii)$'$ (iii) and assume that $a(D'), a(D'')>0$, $$\Delta(D')/a(D')\sim\Delta(D'')/a(D'').$$
Indeed, it is sufficient to show that
\begin{equation}\label{a(D)} \hbox{ if conditions (i)(ii)$'$ (iii)hold and $a(D')=1$, then } a(D'')\Delta(D')\sim\Delta(D''). \end{equation}
This can be proved by induction on $a(D'')$. The case $a(D'')=0$ is trivial since in this case $\Delta(D'')=0$. We have already shown the case $a(D'')=1$. Now by inductive assumption we assume that (\ref{a(D)}) is true for $a(D'')=k-1$ and $k$. Suppose $a(D'')=k+1$. Take $D\in\mathfrak{D}$ such that $D^\nwarrow\sim D''$. (This is always possible by using Transfactor Lemma and Minors Permuting Lemma to modify $D''$.) Then Lemma \ref{lem:powerful} asserts that
$2\Delta(D)\sim\Delta(D^\nwarrow)+\Delta(D^\searrow)$.
By inductive assumption $\Delta(D)\sim k\,\Delta(D')$ and $\Delta(D^\searrow)\sim (k-1)\Delta(D')$, therefore
$$\Delta(D'')\sim \Delta(D^\nwarrow)\sim 2k\,\Delta(D')-(k-1)\Delta(D')=(k+1)\Delta(D'),$$ this completes the inductive proof of (\ref{a(D)}).

As an immediate consequence, any minimal staircase form of partition type $\mu$ generates all the minimal staircase forms of the same partition type $\mu$, modulo $I_{<d}+($minimal staircase forms of partition type $<_P \mu)$. This completes the proof.
\end{proof}

Now we can prove Proposition~\ref{staircaseformofnonminimal}.

\begin{proof}[Proof of Proposition~\ref{staircaseformofnonminimal}] Assume $D=\{P_1,\dots,P_n\}\in \mathfrak{D}$ and $S$ is a staircase form of $D$ and is not minimal. By Transfactor Lemma and Minors Permuting Lemma, we can assume without loss of generality that, in  the block diagonal form $B(S)=\hbox{diag}(B_1,\dots,B_s)$, all the  size-1 blocks stand before the blocks of size greater  than 1.

First note that if the assumption of Lemma~\ref{lem:powerful} is satisfied and the last block of $B(S)$ is not minimal, the conclusion easily follows. Indeed, in this case the claim $\Delta(D)\sim \Delta(D')$ in Lemma~\ref{lem:powerful} implies that we may move any point $P_i$ in the last block of $B(S)$ to $P_i+(1,-1)$. Start from a point $P_i$ for some $i$ that $n-t_0+1\le i\le n-1$ such that it has the same degree as $P_{i+1}$. Keep on moving $P_i$ to southeast direction until it collides with $P_{i+1}$ and then the determinant will be 0.

Now we show that we can always assume the assumption of Lemma~\ref{lem:powerful} is satisfied and the last block of $B(S)$ is not minimal. The
assumption of Lemma~\ref{lem:powerful} is always satisfied by using Minors Permuting Lemma and Transfactor Lemma, since there are sufficient size-1
blocks in $B(S)$. To finish the proof of the proposition, we only need
to exclude the case when the last block $B_s$ of $B(S)$ is minimal. Denote the
size of $B_s$ by $t_0\ge 2$. Define $D_|$ to be the set $\{P_1,\dots,P_{n-t_0}\}$,
define $n'=n-t_0$, and let $d',d_1',d_2',k'$ be the total degree,
x-degree, y-degree and deficit of $D_|$, respectively. Then $k\ge k'+t_0-1$
so
$$\aligned &n'\ge 8k+5-t_0\ge 8(k'+t_0-1)+5-t_0\ge 8k'+5,\\
&d'_1>d_1-t_0n\ge (2k+1)n-t_0n\ge(2k'+t_0-1)n\ge(2k'+1)n\ge
(2k'+1)n',\endaligned$$ and similarly $d_2'\ge (2k'+1)n'$. Then we
can use induction to assert that $\Delta(D_|)$ is generated by
elements in $I_{<d'}$ and minimal staircase form of degree $d'$. Now
$\Delta(D)=\Delta(D_|)\cdot\det(B_s)$ is generated by $I_{<d}$ and
minimal staircase form of degree $d$. Hence in the case when $B_s$
is minimal, there is nothing to prove.
\end{proof}

In order to complete the proof of Theorem~\ref{mainthm}, we use the following Lemma. Recall that
$$\aligned
&a(\lambda)=\sum_i (n-i-\lambda_i),\\
&b(\lambda)=\#\{i<j\, :\,\lambda_i-\lambda_j+i-j\in\{0,1\}\}.\endaligned
$$

\begin{lem}\label{greaterthanpartition}
Let $u$ be an positive integer that $k\le u\le n-2$ and let $v=n-1-u$. If $d_1=u(u+1)/2$ and $d_2=v(v+1)/2+uv-k$, then $\emph{dim }M_{d_1,d_2}\geq p(k)$.
\end{lem}
\begin{proof}
Consider a partition
$$
\lambda=(u+\varepsilon_0,\,\, u-1+\varepsilon_1,\,\, u-2+\varepsilon_2,\,\dots,\, 1+\varepsilon_{u-1},\,0,\,0,\,...,\,0),
$$
where $(v+1)$ zeroes are at the end. If
\begin{equation}\label{zeroonesystem}
\left\{\aligned
&\varepsilon_i=0 \text{ or }1\text{ for }0\leq i\leq u-1,\\
&\sum_{i=0}^{u-1} \varepsilon_i=k,\\
&\text{and }\sum_{i=0}^{u-1} i\varepsilon_i=k(k+1)/2,
\endaligned\right.\end{equation} then it is straightforward to check that $\lambda$ satisfies
$$
b(\lambda)=u(u+1)/2,\,\,\,\text{and}\,\,\,a(\lambda)=v(v+1)/2+uv-k.
$$
Since there are $p(k)$ number of solutions for the system $(\ref{zeroonesystem})$, we have $\text{dim }M_{d_1,d_2}\geq p(k)$.
\end{proof}

Now we are ready to prove Theorem~\ref{mainthm}.

\begin{proof}[Proof of Theorem~\ref{mainthm}] Corollary \ref{lessthanpartitionnumber} asserts
that $M_{d_1,d_2}$ is generated by $p(k)$ elements
$$\{\det(S_\mu)\}_{\mu\in\Pi_k}$$ where $S_\mu$ is an arbitrary
minimal staircase form of bidegree $(d_1,d_2)$ and of partition type
$\mu$. So to complete the proof,  we need to show the minimality of
the above generators, which is equivalent to show $\dim
M_{d_1,d_2}\ge p(k)$.  Lemma~\ref{greaterthanpartition} provides
such a lower bound of  $\dim M_{d_1,d_2}$ for special values of
$d_1$ and $d_2$. For general values of $d_1$ and $d_2$, the idea is
to add sufficiently many appropriate size-1 blocks such that we can
apply Lemma~\ref{greaterthanpartition}. We shall explain as below.

Choose a sufficiently large number $\tilde{n}\gg n$ such that there are positive integers $u$ and $v$ satisfying $k\le u\le \tilde{n}-2$, $1+u+v=\tilde{n}$, $u(u+1)/2\ge (2k+1)\tilde{n}$, and $v(v+1)/2+uv-k\ge (2k+1)\tilde{n}$. Choose $(\tilde{n}-n)$ points $P_i=(\alpha_{i},\beta_{i})$ for $n+1\leq i\leq \tilde{n}$ so that
$$\aligned
&\alpha_{i}+\beta_{i}=i-1 \,\,\,\,(n+1\leq i\leq \tilde{n}),\\
&\sum_{i=1}^{\tilde{n}} \alpha_i =u(u+1)/2=:\tilde{d}_1,\\
&\sum_{i=1}^{\tilde{n}} \beta_i =  v(v+1)/2+uv-k=:\tilde{d}_2
\endaligned$$ which is always possible.
By our choice of $P_i$ ($n+1\le i\le \tilde{n}$), if $D=\{P_1,\dots,P_n\}$ has a minimal staircase form of partition type $\mu$,
then $\tilde{D}=\{P_1,\dots,P_n, P_{n+1},\dots, P_{\tilde{n}}\}$
also has a minimal staircase form of the same partition type $\mu$. Let $\tilde{S}$ be the staircase form of $\tilde{D}$ and $B(\tilde{S})$ the block diagonal form of $\tilde{S}$. Denote by $f_0$ the product of the last $(\tilde{n}-n)$ size-1 minors in $B(\tilde{S})$.
Let $\tilde{I}=\cap_{1\le i<j\le \tilde{n}}(x_i-x_j,y_i-y_j)$ be an ideal of $\C[x_1,y_1,\dots,x_{\tilde{n}}, y_{\tilde{n}}]$, define $\tilde{M}=\tilde{I}/(\textbf{x},\textbf{y})\tilde{I}$ which is doubly graded as $\oplus_{\tilde{d}_1,\tilde{d}_2}\tilde{M}_{\tilde{d}_1,\tilde{d}_2}$. Then we have a $\C$-linear map:
$$\aligned L: M_{d_1,d_2}&\to & \tilde{M}_{\tilde{d}_1,\tilde{d}_2}\\
f&\mapsto &f\cdot f_0.\endaligned$$
For every partition $\mu$ of $k$,  $L(\det S_\mu)$ is of partition type $\mu$. Since $\{L(\det S_\mu) \}_{\mu\in\Pi(k)}$ form a basis for $\tilde{M}_{\tilde{d}_1,\tilde{d}_2}$, the map $L$ is surjective. Therefore  $\dim M_{d_1,d_2}\ge \dim \tilde{M}_{\tilde{d}_1,\tilde{d}_2}\ge p(k)$, which provides the expected lower bound  for $\dim M_{d_1,d_2}$.
\end{proof}

\section{Conjectural set of generators}
Recall that $\Lambda$ is the set of integer sequences $\lambda_1\ge ... \lambda_{n-1}\ge\lambda_n= 0$ satisfying $\lambda_i\le n-i$ for all $i$. We propose the following conjecture.
\begin{conj}\label{conj} For any $\lambda\in\Lambda$, let
$$a_i=n-i-\lambda_i,\quad b_i=\#\{ i < j : \lambda_i-\lambda_j + i-j \in \{0, 1\} \}$$
and $D(\lambda)=\{(a_i, b_i)|1\le i\le n\}$.
 Then $G:=\{\Delta(D(\lambda))\}_{\lambda\in\Lambda}$  generates $I$.
\end{conj}
The cases for $n\le 8$ have been verified by computer. In our forthcoming paper, we will show that this conjecture holds true for certain bi-degree spaces $M_{d_1,d_2}$.


\begin{remark}
An equivalent conjecture is given by Mahir Can and Nick Loehr in their unpublished work.
\end{remark}


\begin{thebibliography}{99}
\bibitem{BC}
N. Bergeron and Z. Chen,  Basis of Diagonally Alternating Harmonic
Polynomials for low degree, arXiv: 0905.0377.

\bibitem{GH:pos}
A. M. Garsia and J. Haglund, A positivity result in the theory of Macdonald polynomials, Proc. Natl. Acad. Sci. USA \textbf{98} (2001), no. 8,
4313--4316 (electronic).

\bibitem{GH:proof}
A. M. Garsia and J. Haglund, A proof of the q, t-Catalan positivity conjecture, Discrete Math. \textbf{256} (2002), no. 3, 677�717, LaCIM 2000 Conference
on Combinatorics, Computer Science and Applications (Montreal, QC).


\bibitem{H:hil}
Mark Haiman, Hilbert schemes, polygraphs and the Macdonald
positivity conjecture, J. Amer. Math. Soc. \textbf{14} (2001), no.
4, 941--1006.

\bibitem{van}
Mark Haiman, Vanishing theorems and character formulas for the Hilbert scheme of points in the plane, Invent. Math. \textbf{149} (2002), no. 2, 371--407.

\bibitem{H04}
Mark Haiman, Commutative algebra of $n$ points in the plane, With an appendix by Ezra Miller. Math. Sci. Res. Inst. Publ., 51,  Trends in commutative algebra,  153--180, Cambridge Univ. Press, Cambridge, 2004.

\bibitem{hardy}
G.H. Hardy, Ramanujan: twelve lectures on subjects suggested by his life and work, Chelsea Publishing Company, New York 1959.

\bibitem{LL2}
K. Lee and L. Li, $q,t$-Catalan numbers and generators for the radical ideal defining the diagonal locus of $(\C^2)^n$.
\end{thebibliography}
\end{document}